\documentclass[11pt]{article}
\usepackage{amssymb}
\usepackage{amsmath}
\usepackage{makeidx}
\usepackage{amsfonts}

\newtheorem{theorem}{Theorem}

\newtheorem{lemma}{Lemma}
\newtheorem{definition}{Definition}

\newtheorem{remark}{Remark}

\newenvironment{proof}[1][Proof]{\noindent\textbf{#1.} }{\ \rule{0.5em}{0.5em}}
\numberwithin{equation}{section}
\input{tcilatex}

\begin{document}

\title{Exact null-controllability of interconnected abstract evolution
systems by scalar force motion}
\author{B. Shklyar
(Holon Institute of Technology, Holon, Israel)}
\date{}
\maketitle

\begin{abstract}
The paper deals with exact null-controllability problem for a linear control
system consisting of two serially connected abstract control systems.
Controllability conditions are obtained. Applications to the exact
null-controllability for interconnected control system of heat and wave
equations are considered.
\end{abstract}

strongly minimal families.


\section{Introduction and problem statement}

Many engineering applications generate interactive physical processes
described by interconnected control systems. Control design for such systems
modeled by interconnected partial differential control systems,~have
investigated intensively over the last years.

The goal of the present paper is to establish complete null controllability
conditions for a control object containing two control abstract evolution
systems interconnected into a series such that a control function from the
second control system is an output of the first one.

Let $X_1, X_2$ be complex separable Hilbert spaces. Consider the control
evolution equation \cite{Hille&Philips},\cite{Krein} with scalar control
\begin{align}
\dot{x}_{1}\left( t\right) & =A_{1}x_{1}\left( t\right) +b_{1}v\left(
t\right) ,x_{1}\left( 0\right) =x_{1}^{0},  \label{1.1} \\
v\left( t\right) & =\left( c,x_{2}\left( t\right) \right) ,0\leq t<+\infty ,
\label{1.2}
\end{align}%
where $x_{2}(t)$ is a mild solution of the another control equation of the
form%
\begin{equation}
\dot{x}_{2}\left( t\right) =A_{2}x_{2}\left( t\right) +b_{2}u\left( t\right)
,0\leq t<+\infty ,~x_{2}\left( 0\right) =x_{2}^{0}.  \label{1.3}
\end{equation}%
Here $x_{1}\left( t\right) ,x_{1}^{0},b_{1}\in X_{1},$ where $X_{1}$ is the
state space of equation (\ref{1.1}), $v\left(t\right) \in \mathbb{C},
x_{2}\left( t\right) ,x_{2}^{0},c,b_{2}\in X_{2},$ where $X_{2}$ is the
state space of equation (\ref{1.3}), $u\left(t\right) \in \mathbb{C}$, and
the linear operators $A_{1}$ and $A_{2}$ generate strongly continuous $C_{0}$%
-semigroup $S_{1}\left( t\right) $ in $X_{1}$ and $S_{2}\left( t\right) $ in
$X_{2}$ correspondingly \cite{Hille&Philips, Krein}.

The formulas $b_{1}v$ and $b_{2}u, v,u\in \mathbb{C}$ express linear bounded
operators from $\mathbb{C}$ to $X_{1}$ and from $\mathbb{C}$ to $X_{2}.$

The interconnected system (\ref{1.1})--(\ref{1.3}) is governed by a control $%
u\left( t\right) $ of equation (\ref{1.3}).

\section{Basic assumptions and definitions}

\label{Basic assumptions}


\begin{enumerate}
\item The operator $A_{1}$ has purely point spectrums $\sigma _{1}$ with no
finite limit points. 

Since we use scalar controls we assume the geometrical multiplicity of
eigenvalues of the operator $A_{1}$ to be equal to 1.

\item All eigenvectors of the operators $A_{1}$ produce a Riesz basic in
their linear span.
\end{enumerate}

Let the eigenvalues $\lambda _{j} \in \sigma _{1},j=1,2,\ldots ,$ be
enumerated in the order of non-decreasing absolute values, let $\alpha _{j}$
be the algebraic multiplicities \footnote{%
The geometric multiplicity is the number of Jordan blocks corresponding to $%
\lambda _{j}\in \sigma _{1}$. Throughout in the paper it is equal to $1$.}
of $\lambda _{j}\in \sigma _{1}$ correspondingly, and let $\varphi
_{jk},j\in {\mathbb{N}},k=1,2,\ldots ,\alpha _{j}$ be the generalized
eigenvectors of the operator $A_1, A_1\varphi _{j\alpha _{j}}=\lambda
_{j}\varphi _{j\alpha _{j}}, j\in {\mathbb{N}}$, and let $\psi _{jk}, j\in {%
\mathbb{N}},k=1,2,\ldots ,\alpha _{j}$, be the generalized eigenvectors of
the adjoint operator $A_1^{\ast },A_1^{\ast }\psi _{j\alpha _{j}}=\bar{%
\lambda}_{j}\psi _{j\alpha _{j}},j\in {\mathbb{N}},$ chosen such that
\begin{eqnarray}
(\varphi _{s\alpha _{s}-l+1,} \psi _{jk}) &=&\delta _{sj}\delta _{lk},
\label{2.0} \\
s,j &\in &{\mathbb{N}},\,\,\,l=1,\ldots ,\alpha _{s}
,\,\,\,\,k=1,\ldots,\alpha _{j}.  \notag
\end{eqnarray}%
We use the following notations\footnote{%
If $0\in \sigma _{1},$ we denote $\lambda _{0}=0$ and in (\ref{2.0})--(\ref%
{2.4}) $j\in 0\cup \mathbb{N}.$}:
\begin{align}
x_{jk}\left( t\right) & =\left( x\left( t,x_{0},u\left( \cdot \right)
\right) ,\psi _{jk}\right) ,\,\,x_{jk}^{0}=\left( x_{1}^{0},\psi
_{jk}\right),  \label{2.01} \\
b_{1jk}& =\left( b_{1},\psi _{jk}\right) ,j\in {\mathbb{N}},k=1,2,\ldots
\alpha _{j},  \notag
\end{align}

\begin{align}
g_{jk}(-t)& =e^{-\lambda _{1j}t}\sum_{l=0}^{\alpha _{j}-k}b_{1jk+l}{\frac{%
(-t)^{l}}{l!}},\,\,t\in \lbrack 0,t_{1}],  \label{2.1} \\
j& \in {\mathbb{N}},\,\,\,k=1,2,\ldots \alpha _{j}.  \notag
\end{align}

The following properties of sequences $\left\{x_j\in X_1,j=1,2,\dots
\right\} $ are very significant throughout in the given paper.

\begin{definition}
\label{D2.1} The sequence $\left\{ x_{j}\in X_1,j=1,2,\dots \right\} $ is
said to be minimal, if there no element of the sequence belonging to the
closure of the linear span of others. By other words,
\begin{equation*}
x_{j}\notin \overline{\mathrm{span}}\left\{ x_{k}\in X_1,k=1,2,\dots ,k\neq
j\right\} .
\end{equation*}
\end{definition}

\begin{definition}
\label{D2.2}The sequence $\left\{ x_{j}\in X_1,j=1,2,\dots \right\} $ is
said to be strongly minimal, if there exists a positive number $\gamma >0$
such that
\begin{equation}
\gamma \sum_{k=1}^{n}\left\vert c_{k}\right\vert ^{2}\leq \left\Vert
\sum_{k=1}^{n}c_{k}x_{k}\right\Vert ^{2},\,\,n=1,2,...,  \label{2.2}
\end{equation}%
where
\begin{equation*}
\gamma =\lim_{n\rightarrow \infty }\min_{\substack{ c_{1},..,c_{n}:  \\ %
\sum_{k=1}^{n}\left\vert c_{k}\right\vert ^{2}=1}}\left\Vert
\sum_{k=1}^{n}c_{k}x_{k}\right\Vert ^{2}.
\end{equation*}
\end{definition}

Using above properties the following results have been proven in \cite%
{Shklyar2011}.

\begin{definition}
\label{D3.0} Equation (\ref{1.1}) is said to be exact null-controllable on $%
\left[ 0,t_{1}\right] $ by square integrable controls, if for each $%
x_{10}\in \mathfrak{X}_{1}$ and $\alpha \in \mathbb{R}$ there exists a
control $u\left( \cdot \right) \in L_{2}\left[ 0,t_{1}\right] $, such that
\begin{equation}
x_{1}\left( t_{1},x_{10},v\left( \cdot \right) \right) =0.  \label{2.2.1}
\end{equation}
\end{definition}

\begin{theorem}
\label{T2.1}\cite{Shklyar2011}~Let the sequence of eigenvectors of operator $%
A_{1}$ forms a Riesz basic in $X_{1}$. Equation (\ref{1.1}) is exact
null-controllable on $\left[ 0,t_{1}\right] ~$by scalar controls, if and
only if the sequence (\ref{2.1}) of generalized exponents is strongly
minimal in $L_{2}([0,t_{1}])$.
\end{theorem}

For the simplicity of the exposition we assume below that all the
eigenvalues of the operator $A_{1}$ are simple. In this case the eigenvector
of the operator $A_{1}^{\ast },$ corresponding to the eigenvalue $\lambda
_{j},$ can be denoted by $\psi _{j},~j=1,2,...,$ $b_{1j}=\left( b_{1},\psi
_{j}\right) ,~j=1,2,...$and the family of generalized exponents (\ref{2.1})
can be simplified and written by exponents

\begin{equation}
\left\{ g_{j}\left( -t\right) =b_{1j}e^{-\lambda _{1j}t}, j=1,2,...\right\} .
\label{2.3}
\end{equation}

If $0\in \sigma _{1},$ then according to our assumption $0$ is a simple
eigenvalue, and $\sigma _{1}=\left\{ \lambda _{j},j=0,1,2,...,\right\} ,$
where $\lambda _{0}=0$. Otherwise $\sigma _{1}=\left\{ \lambda _{j},j=\
1,2,...,\right\} .$ In both cases $\lambda _{j}\neq 0,~j=1,2,...,,$ and

\begin{equation}
\left\{ g_{j}\left( -t\right) =b_{1j}e^{-\lambda _{j}t}~,j=0,1,2,...\right\}
\label{2.4}
\end{equation}

\section{Controllability of interconnected equations by force motions}

The problem can be investigated by two different ways.

The first way is :

\noindent
we construct a composite evolution equation
\begin{eqnarray}
\dot{x} &=&\mathtt{A}x\left( t\right) +\mathtt{b}u\left( t\right) ,
\label{2.5} \\
x\left( 0\right) &=&x_{0},  \notag
\end{eqnarray}%
where $x\left( t\right) =\left(
\begin{array}{c}
x_{1}\left( t\right) \\
x_{2}\left( t\right)%
\end{array}%
\right) \in X=X_{1}\times X_{2},~x\left( 0\right) =\left(
\begin{array}{c}
x_{10}\  \\
x_{20}%
\end{array}%
\right) ,$

$\mathtt{A}=\left(
\begin{array}{cc}
A_{1} & B \\
0 & A_{2}%
\end{array}%
\right) ,$ where the linear operator $B:X_{2}\rightarrow X_{1}$ is defined
by $Bx_{2}=b_{1}\left( c,x_{2}\right) ,$ $\mathtt{b}=\left(
\begin{array}{c}
0 \\
b_{2}%
\end{array}%
\right) \in X.$

System (\ref{2.5}) is considered as a system combining the features of both
equations (\ref{1.1}) and (\ref{1.3}).

Next need to prove that system (\ref{2.5}) is an equation of the form (\ref%
{1.1}), i.e. the operator $\mathtt{A}$ generates $C_{0}$-semigroup and
satisfies the conditions imposed on the operator $A_{1}$ (see page \ref%
{Basic assumptions}), and afterward one can use known controllability
conditions of equation (\ref{1.1})

As a rule this way has been used in the literature for the case when both equations
(\ref{1.1}) and (\ref{1.3}) are control PDE's.

Since the linear operator $Bx_{2}=b_{1}\left( c,x_{2}\right) $ is obviously
bounded one can prove that the operator $\mathtt{A}$ generates $C_{0}$%
-semigroup.

In order to continue we need to prove that assumptions on page \ref{Basic
assumptions} hold. It may be done, if the operator $A_{2}$ satisfies the
same conditions on page \ref{Basic assumptions}. However in this paper we
know nothing about the operator $A_{2}$ except that the operator $A_{2}$
generates a $C_{0}$-semigroup.

It gives the motivation to use the different approach.

\subsection{Controllability of equation (\protect\ref{1.1}) by smooth
controls}

Let $\ AC\left[ 0,t_{1}\right] $ be the space of absolutely continuous
functions defined on the closed segment $\left[ 0,t_{1}\right] ,~$and$\ $let$%
~a\in \mathbb{C}$. Denote:~
\begin{eqnarray*}
H_{\alpha }^{1}\left[ 0,t_{1}\right] &=&\left\{ v\left( \cdot \right) \in \
AC\left[ 0,t_{1}\right] ,~v(0)=a,v^{\prime }\left( \cdot \right) \in L_{2}%
\left[ 0,t_{1}\right] \right\} , \\
H_{\alpha \beta }^{1}\left[ 0,t_{1}\right] &=&\left\{ v\left( \cdot \right)
\in \ AC\left[ 0,t_{1}\right] ,~v(0)=a,v\left( t_{1}\right) =\beta
,~v^{\prime }\left( \cdot \right) \in L_{2}\left[ 0,t_{1}\right] \right\} ,
\\
H_{0}^{1}\left[ 0,t_{1}\right] &=&H_{00}^{1}\left[ 0,t_{1}\right]
\end{eqnarray*}

We know nothing about differential properties of a generalized solution $%
x_{2}\left( t\right) ,$ generated by control $u\left( \cdot \right) \in L_{2}%
\left[ 0,t_{1}\right] ,$ but in accordance with the definition of a
generalized solution of equation (\ref{1.3}) the function $v\left( t\right)
=\left( c,x_{2}\left( t\right) \right) $ defined by (\ref{1.2}) is
absolutely continuous for any $c\in D\left( A_{2}^{\ast }\right) \ $ \cite%
{Balakrishnan}. In order to keep the control object in the equilibrium
state, we will turn off the control $v\left( t\right) $ at the end of the
control process, i.e. $v\left( t\right) \equiv 0,t\geq t_{1}.$ Hence to
investigate the exact null-controllability of interconnected equations (\ref%
{1.1})-(\ref{1.3}) it makes sense to consider the exact null controllability
of equation (\ref{1.1}) on $\left[ 0,t_{1}\right] \ $ by single smooth
controls $v\left( \cdot \right) $ of the space $H_{\alpha }^{1}\left[ 0,t_{1}%
\right] $ .

\begin{definition}
\label{D3.3}Equation (\ref{1.1}) is said to be exact null-controllable on $%
\left[ 0,t_{1}\right] $ by smooth controls, if for each $x_{10}\in \mathfrak{%
X}_{1}$ and $\alpha \in \mathbb{C}$ there exists a control $v\left( \cdot
\right) \in H_{\alpha 0}^{1}\left[ 0,t_{1}\right] $, such that
\begin{equation}
x_{1}\left( t_{1},x_{10},v\left( \cdot \right) \right) =0.  \label{3.0}
\end{equation}
\end{definition}

To establish the controllability conditions by smooth controls, we need the
following auxiliary result.

\begin{lemma}
\label{L3.1} The operator $\mathcal{A=}\left(
\begin{array}{cc}
A_{1} & b_{1} \\
0 & 0%
\end{array}%
\right) \mathcal{\ }$ generates strongly continuous $C_{0}$-semigroup in the
product space $X_{1}\times \mathbb{C}.$
\end{lemma}

\textbf{Proof. }Denote by $R$ $\left( \mu \right) $and $\mathcal{R}\left(
\mu \right) $ the resolvent operators of the operators $A\ $ and $\mathcal{A}
$ correspondingly.

Denote by $\mathcal{R}_{0}\left( \mu \right)$ the resolvent of the operator $%
\mathcal{A}_{0}\mathcal{=}\left(
\begin{array}{cc}
A_{1} & 0 \\
0 & 0%
\end{array}%
\right).$ Obviously
\begin{equation*}
\left\Vert \mathcal{R}_{0}^{n}\left( \mu \right) \left(
\begin{array}{c}
x \\
v%
\end{array}%
\right) \right\Vert ^{2}=\left\Vert R^{n}\left( \mu \right) x\right\Vert
^{2}+\left\Vert \frac{1}{\mu ^{n}}v\right\Vert ^{2},\forall n\in \mathbb{N}.
\end{equation*}

In accordance with Hille-Iosida Theorem \cite{Hille&Philips} there exist
positive constants $K,a,$such that~
\begin{equation*}
\left\Vert R^{n}\left( \mu \right) \right\Vert \leq \frac{K_{1}}{\left( \mu
-a\right) ^{n}}\,\,\forall \mu >a,\,\,\forall n\in \mathbb{N},
\end{equation*}%
Therefore
\begin{eqnarray*}
\left\Vert \mathcal{R}_{0}^{n}\left( \mu \right) \left(
\begin{array}{c}
x \\
v%
\end{array}%
\right) \right\Vert ^{2} &\leq &\frac{K^{2}}{\left( \mu -a\right) ^{2n}}%
\left\Vert x\right\Vert ^{2}+\frac{1}{\mu ^{2n}}\left\Vert v\right\Vert
^{2}\leq \\
&\leq &\frac{K_{1}^{2}}{\left( \mu -a\right) ^{2n}}\left( \left\Vert
x\right\Vert ^{2}+\left\Vert v\right\Vert ^{2}\right) , \\
\forall \mu &>&a,\forall n\in \mathbb{N}.
\end{eqnarray*}%
where $K_{1}=\max \{K,1\}.$Hence

\begin{equation*}
\left\Vert \mathcal{R}_{0}^{n}\left( \mu \right) \right\Vert \leq \frac{K_{1}%
}{\left( \mu -a\right) ^{n}},\,\,\forall n\in \mathbb{N},\forall \mu >a.
\end{equation*}%
It shows~\cite{Hille&Philips} that the operator $\mathcal{A}_{0}$ generates $%
C_{0}$-semigroups.

Denote $\mathcal{B}$ $=\left(
\begin{array}{cc}
0 & b_{1} \\
0 & 0%
\end{array}%
\right) .$ Obviously the operator $\mathcal{B}$ is bounded and $\mathcal{A=A}%
_{0}+\mathcal{B}.$

It is well-known \cite{Hille&Philips}, that if the operator $\mathcal{A}_{1}$
generates $C_{0}$-semigroup and the operator $\mathcal{B}$ is bounded, then
the operator $\mathcal{A}=\mathcal{A}_1+\mathcal{B}$ also generates $C_{0}$%
-semigroup.

It proves the lemma.

\textbf{Remark. } Obviously, if equation (\ref{1.1}) is exact
null-controllable on $\left[ 0,t_{1}\right] $ by smooth controls $v\left(
\cdot \right) \in H_{a0}^{1}\left[ 0,t_{1}\right] ,$ then, obviously, it is
exact null-controllable on $\left[ 0,t_{1}\right] $ by $v\left( \cdot
\right) \in L_{2}\left[ 0,t_{1}\right]$. According to Theorem \ref{T2.1} the
family (\ref{2.4}) of exponents should be strongly minimal. Surely it is
impossible, if $0\in \sigma _{1}$ and $b_{10}=0.$ Hence in the case of $0\in
\sigma _{1}$ it makes sense to consider only the condition $b_{10}\neq 0. $

\begin{theorem}
\label{T3.1}Equation (\ref{1.1}) is exact null-controllable on $\left[
0,t_{1}\right] $ by smooth controls $v\left( \cdot \right) \in H_{a0}^{1}%
\left[ 0,t_{1}\right] $, if and only then either $0\notin \sigma _{1}$
or $0\in \sigma _{1}$ and $b_{10}\neq 0$, and family
\begin{equation}
\left\{ \ 1,\frac{b_{1j}}{\lambda _{j}}e^{-\lambda _{1j}t},j=1,2,...\right\}
\label{3.1}
\end{equation}
of exponents is strongly minimal. 
\end{theorem}

\begin{proof}
One can write system (\ref{1.1})) governed by smooth control $v\left( \cdot
\right) \in H_{\alpha }^{1}\left[ 0,t_{1}\right] ,$ by
\begin{align}
\dot{x}_{1}\left( t\right) & =A_{1}x_{1}\left( t\right) +b_{1}v\left(
t\right) ,x_{1}\left( 0\right) =x_{1}^{0},  \label{3.1.1} \\
\dot{v}\left( t\right) & =u\left( t\right) ,v\left( 0\right) =\alpha ,~0\leq
t<+\infty ,  \label{3.1.2}
\end{align}%
As far as the operator $\mathcal{A}$ generates a strongly continuous $C_{0}$%
-semigroup, composite system (\ref{3.3.1})--(\ref{3.1.2}) can be written in
the form of (\ref{1.1}) as follows:
\begin{equation}
\dot{z}\left( t\right) =\mathcal{A}z\left( t\right) +\mathfrak{b}v\left(
t\right) ,  \label{3.2}
\end{equation}%
\bigskip where $z=\left(
\begin{array}{c}
x \\
v%
\end{array}%
\right) \in X_{1}\times \mathbb{C},~$the operator $\mathcal{A}$~is defined
in Lemma \ref{3.1}, $\mathfrak{b}=\left(
\begin{array}{c}
0 \\
1%
\end{array}%
\right) .$

Let $\sigma ({\mathcal{A}})$ be the spectrum of the operator $\mathcal{A}.$
We have $\sigma ({\mathcal{A}})=\sigma _{1}\cup \{0\},$ $0\notin \sigma .$
One can see that the operator $\mathcal{A}$ satisfies the assumption on page %
\ref{Basic assumptions}

Denote by $\mathcal{\psi }=\left(
\begin{array}{c}
\psi ^{1} \\
\psi ^{2}%
\end{array}%
\right) \in X\times \mathbb{C}$ the eigenvector of the adjoint operator $%
\mathcal{A}^{\ast }$ corresponding to the eigenvalue $\lambda \in \sigma ({\
\mathcal{A}}).$ We have
\begin{equation*}
\mathcal{A}^{\ast }\mathbb{=}\left(
\begin{array}{cc}
A_{1}^{\ast } & 0 \\
b_{1}^{\ast } & \ 0%
\end{array}%
\right) ,
\end{equation*}%
where $b_{1}^{\ast }$ is a linear functional from $X_{1}~$to $\mathbb{C},~$
defined by 
\begin{equation*}
b_{1}^{\ast }x=\left( b_{1},x\right) ,\forall x\in X_{1}.
\end{equation*}

The eigenvalues and corresponding eigenvectors of the operator $\mathcal{A}%
^{\ast }$ are defined as follows:
\begin{equation}
\left( \overline{\lambda }I-\mathcal{A}^{\ast }\right) \mathcal{\psi }%
=\left(
\begin{array}{c}
\left( \overline{\lambda }I-A_{1}^{\ast }\right) \psi ^{1} \\
-b^{\ast }\psi ^{1}+\overline{\lambda }\psi ^{2}%
\end{array}%
\right) =\left(
\begin{array}{c}
0 \\
0%
\end{array}%
\right) .  \label{3.3}
\end{equation}%
Equality (\ref{3.3}) holds if and only if
\begin{equation}
\left( \overline{\lambda }I-A^{\ast }\right) \psi ^{1}=0,-b^{\ast }\psi ^{1}+%
\overline{\lambda }\psi ^{2}=0.  \label{3.3.01}
\end{equation}

If $\lambda \in \sigma _{1},$ and $\lambda \neq 0,~$ then from (\ref{3.3})
it follows, that if $\psi ^{1}=0,~$then $\psi ^{2}=0$ as well, but it is
impossible, because $\mathcal{\psi }$ is an eigenvector. Hence $\psi ^{1}~$%
is an eigenvector of $A_{1}^{\ast }$ and in accordance with the theorem
conditions and $\psi ^{2}=\frac{1}{{\lambda }}\left( b_{1},\psi ^{1}\right)
. $ In this case the eigenvectors $\mathcal{\psi }_{\lambda }$ of the
operator $\mathcal{A}^{\ast }$ corresponding to its eigenvalue $\lambda \in
\sigma _{1},$ $\lambda \neq 0$ are defined as follows:
\begin{equation*}
\mathcal{\psi }_{\lambda }=\left(
\begin{array}{c}
\psi ^{1} \\
\frac{1}{\lambda }\left( b,\psi ^{1}\right)%
\end{array}%
\right) ,
\end{equation*}%
where $\psi ^{1}$ is an eigenvector of $A_{1}^{\ast }$, corresponding to an
eigenvalue $\lambda \in \sigma .$

\textbf{1}. Let's continue to prove the theorem for the case $0\notin \sigma
_{1}$.

If $\lambda =0,$ then $~$ $0$ is a regular value for $A_{1},$ so from (\ref%
{3.3.01}) it follows, that $\psi ^{1}=0$, and therefore $\psi ^{2}$ may be
any nonzero constant. One can set $\psi ^{2}=1$. Therefore in this case the
eigenvalues and corresponding eigenvectors of the operator $\mathcal{A}%
^{\ast }$ are defined as follows:
\begin{equation*}
\mathcal{\psi }_{0}=\left(
\begin{array}{c}
0 \\
1%
\end{array}%
\right) .
\end{equation*}

Summary: let $\lambda _{j},$ $j=0,1,2,\dots $ $\in \sigma _{1}\cup \left\{
0\right\} =\sigma \left( \mathcal{A}\right) $ enumerated by increasing of
their absolute values.


The eigenvectors $\mathcal{\psi }_{j},j=0,1,2,...$ of the operator $\mathcal{%
A}^{\ast }$ are
\begin{equation}
\mathcal{\psi }_{j}=\left(
\begin{array}{c}
\psi _{j}^{1} \\
\frac{b_{1j}}{\lambda _{j}}%
\end{array}%
\right) ,\,\,j=1,2,\dots ,\mathcal{\psi }_{0}=\left(
\begin{array}{c}
0 \\
1%
\end{array}%
\right) ,\,  \label{3.3.1}
\end{equation}%
where $\psi _{j}^{1},j=1,2,\dots ,$ are eigenvectors of the operator $%
A^{\ast }$.

From (\ref{3.3.1})~it follows, that
\begin{equation}
\left( \mathfrak{b},\mathcal{\psi }_{j}\right) =\left\{
\begin{array}{cc}
1, & j=0, \\
\frac{b_{1j}}{\lambda _{j}}, & j=1,2,\dots .%
\end{array}%
\right. ~  \label{3.3.2}
\end{equation}

Hence one can see that the sequence (\ref{2.3}) of generalized exponents for
system (\ref{3.2}) is exactly the sequence (\ref{3.1}).

By Theorem \ref{T2.1} the exact null-controllability of system (\ref{3.2})
holds true if and only if family (\ref{3.1}) of exponents is strongly
minimal.

The exact null-controllability of system (\ref{3.2}) is completely
equivalent to the exact null-controllability of equation (\ref{1.1}) on $%
\left[ 0,t_{1}\right] ,$ $t_{1}>0$ by smooth controls.

It proves the theorem for the case $0\notin \sigma _{1}$.

\textbf{2}. Prove the theorem for the case $0\in \sigma _{1}$.

Let $\mathcal{\psi }_{0}=\left(
\begin{array}{c}
\psi _{0}^{1} \\
\psi _{0}^{2}%
\end{array}%
\right) $ be an eigenvector of the operator $\mathcal{A}^{\ast },$
corresponding the eigenvalue $\lambda _{0}=0.$ In accordance with (\ref%
{3.3.01}) from (\ref{3.3}) it follows, that%
\begin{equation}
\left( -A_{1}^{\ast }\right) \psi _{0}^{1}=0,-\left( b_{1},\psi
_{0}^{1}\right) =0..  \label{3.3.3}
\end{equation}
Let $\psi _{0}^{1}\neq 0.$ From (\ref{3.3.3}) it follows, that $\psi _{0}$
is an eigenvector of the operator $A_{1}^{\ast },$ corresponding to the
eigenvalue $\lambda _{0}=0,$ and $b_{10}=\left( b_{1},\psi _{0}^{1}\right)
=0,\psi _{0}^{1}\neq 0.$ This contradicts to theorem conditions.

Hence, $\psi _{0}^{1}=0,$ so again corresponding eigenvector $\psi _{0}$ of
the operator $\mathcal{A}^{\ast }$ is defined as well as in the case $%
0\notin \sigma _{1},$ i.e.
\begin{equation*}
\mathcal{\psi }_{0}=\left(
\begin{array}{c}
0 \\
1%
\end{array}%
\right) ,
\end{equation*}%
Hence again $\left( \mathfrak{b},\mathcal{\psi }_{j}\right) ,~j=1,2,\dots $
is defined by (\ref{3.3.2})

The proof is finished as well as in the case $0\notin \sigma _{1}.$
\end{proof}

\subsection{Controllability criterion of interconnected equations by force
motion}

\begin{definition}
\label{D3.1}Interconnected system (\ref{1.1})--(\ref{1.3}) is said to be
exact null-controllable on $[0,t_{1}]$ if for each $x_{1}^{0}\in
X_{1},~x_{20}\in X_{2}$ there exists a control $u\left( \cdot \right) \in
L_{2}\left[ 0,t_{1}\right] ,$ such that a mild solution\ $x_{1}\left(
t,x_{1}^{0},v\left( \cdot \right) \right) $ of equation (\ref{1.1}) with a
control $v\left( t\right) $ defined by (\ref{1.2}) satisfies the condition
\begin{equation}
x_{1}\left( t,x_{1}^{0},v\left( \cdot \right) \right) =0.  \label{3.4}
\end{equation}
\end{definition}

\subsubsection{Regular case $\left( c,b_{2}\right) \neq 0$}

\begin{theorem}
\label{T3.2}If

\begin{itemize}
\item the family of exponents (\ref{3.1}) is strongly minimal,

\item $c\in D\left( A_{2}^{\ast }\right)$ and $\left( c,b_{2}\right) \neq 0,$
\end{itemize}

then interconnected equation (\ref{1.1})--(\ref{1.3}) is exact
null-controllable on $[0,t_{1}]$.
\end{theorem}

\begin{proof}
Let $c\in D\left( A_{2}^{\ast }\right) $ and $\left( c,b_{2}\right) \neq 0.$
As much as family (\ref{3.1}) of exponents is strongly minimal, then,
according to Theorem \ref{T3.1}, for any $x_{10}\in X_{1}$ and $\alpha \in
\mathbb{C}~\ \ \ $there exists a control \ $v\left( \cdot \right) \in
H_{\alpha 0}^{1}\left[ 0,t_{1}\right] ~$such that (\ref{3.4}) holds, and
vice versa. Hence if any function $v\left( \cdot \right) \in H_{\alpha 0}%
\left[ 0,t_{1}\right] \ ~$can be expressed by

\begin{equation}
v\left( t\right) -v\left( 0\right) =\left( c,x_{2}\left( t\right) \right)
,t\in \left[ 0,t_{1}\right] ,  \label{3.5}
\end{equation}%
where \ \
\begin{equation}
x_{2}\left( t\right) =S_{2}(t)x_{20}+\int_{0}^{t}S_{2}\left( t-\tau \right)
b_{2}u\left( \tau \right) d\tau ,  \label{3.6}
\end{equation}%
then interconnected system (\ref{1.1})--(\ref{1.3}) is exact
null-controllable on $\left[ 0,t_{1}\right] $.

Obviously any function $v\left( \cdot \right) \in H_{a0}^{1}\left[ 0,t_{1}%
\right] $ can be expressed by (\ref{3.5}) if and only if the Volterra
integral equation of the first kind with continuous kernel $K\left( t-\tau
\right) =\left( c,S_{2}\left( t-\tau \right) b_{2}\right) $

\begin{equation}
w\left( t\right) =\int_{0}^{t}\left( c,S_{2}\left( t-\tau \right)
b_{2}\right) u\left( \tau \right) d\tau ,t\in \left[ 0,t_{1}\right] ,
\label{3.7}
\end{equation}%
where $w(t)=v\left( t\right) -v\left( 0\right) ,~v\left( 0\right) =\left(
c,x_{20}\right) $, has a solution \ $u\left( \cdot \right) \in L_{2}\left[
0,t_{1}\right] $ for any $w\left( \cdot \right) \in H_{0}^{1}\left[ 0,t_{1}%
\right] .$


One can use some classical conditions for the existence of solutions for
equation (\ref{3.7}) \cite{Tricomi}. One of them are:

\textbf{1) }$w\left( t\right) $ is continuously differentiable and \ $%
w\left(0\right) =0,$

\textbf{2)} the kernel $K\left( t-\tau \right) =\left( c,S_{2}\left( t-\tau
\right) b_{2}\right) $ of equation (\ref{3.7}) is continuously
differentiable and \ $K\left( t,t\right) =\left( c,b_{2}\right) \neq 0.$

Since $c\in D\left( A_{2}^{\ast }\right) ,$we have$\frac{\partial }{\partial
\tau }K\left( t-\tau \right) =-\left( S_{2}^{\ast }\left( t-\tau \right)
A_{2}^{\ast }c,b_{2}\right) $ to be continuous \cite{Hille&Philips, Krein},
hence if absolutely continuous function $v\left( t\right) $ appears to be
continuously differentiable, there exists a continuous solution of equation (%
\ref{3.7})\cite{Tricomi}.


If absolutely continuous function $v\left( t\right) $ is not continuously
differentiable, we consider the integral Volterra equation of the second kind%
\begin{equation}
w\left( t\right) =v\left( t\right) -v\left( 0\right) =\left( c,b_{2}\right)
U\left( t\right) +\int_{0}^{t}\left( S_{2}^{\ast }\left( t-\tau \right)
A_{2}^{\ast }c,b_{2}\right) U\left( \tau \right) d\tau .  \label{3.8}
\end{equation}

Because of \ $c\in D\left( A_{2}^{\ast }\right) $ the kernel $\ \
K_{1}\left( t-\tau \right) =\left( S_{2}^{\ast }\left( t-\tau \right)
A_{2}^{\ast }c,b_{2}\right) $ is continuous \cite{Hille&Philips, Krein},
hence equation (\ref{3.8}) has a continuous solution $U\left( t\right) $ for
any continuous function \ $v\left( t\right) $ \cite{Tricomi}. This solution
is obtained by \cite{Tricomi}
\begin{equation}
U\left( t\right) =\frac{w\left( t\right) }{\left( A_{2}^{\ast
}c,b_{2}\right) }+\int_{0}^{t}R\left( t-\tau \right) v\left( \tau \right)
d\tau ,  \label{3.8.1}
\end{equation}%
where \ \ $R\left( t\right) ~$is the resolvent of equation (\ref{3.8}),
obtained by \cite{Tricomi}

\begin{equation}
R\left( t-\tau \right) =\sum_{n=0}^{\infty }K_{n+1}\left( t-\tau \right)
,~0\leq \tau \leq t\leq t_{1},  \label{3.8.2}
\end{equation}%
where $K_{n+1}\left( t-\tau \right) $ are repeated kernels defined by the
recurrence%
\begin{equation}
K_{n+1}\left( t-\tau \right) =\int_{\tau }^{t}K_{1}\left( t-\theta \right)
K_{n}\left( \theta -\tau \right) d\tau ,n=1,2,...  \label{3.8.3}
\end{equation}%
and the series (\ref{3.8.1}) converges uniformly. Hence $R\left( t-\tau
\right) $ is continuous, so from (\ref{3.8.1}) it follows, that the function
$U\left( t\right) $ is absolutely continuous function, i.e. there exists an
integrable function $u\left( t\right) ,$ such that $u\left( t\right) =\dot{U}%
\left( t\right) \ \mathrm{a.e.~for}~t\in \left[ 0,t_{1}\right] ,$ and $%
U\left( 0\right) =0$. Actually because of square integrability of $\dot{v}%
\left( t\right) $ the function $u\left( t\right) ~$appears to be square
integrable. Using the integrating by parts we obtain by (\ref{3.4}) and
taking into account the condition $U\left( 0\right) =0$

\begin{eqnarray}
\left( c,b_{2}\right) U\left( t\right) +\int_{0}^{t}\left( S_{2}^{\ast
}\left( t-\tau \right) A_{2}^{\ast }c,b_{2}\right) U\left( \tau \right)
d\tau &=&  \label{3.9} \\
=\int_{0}^{t}\left( c,S_{2}\left( t-\tau \right) b_{2}\right) u\left( \tau
\right) d\tau &=&w\left( t\right) ,  \notag
\end{eqnarray}%
i.e. (\ref{3.7}) holds for a function $u\left( t\right) \in L_{2}\left[
0,t_{1}\right] .$

This proves the theorem.
\end{proof}

\begin{remark}
A square integrable control $u(t)$ is a square integrable
first derivative of an absolutely continuous solution $U\left( t\right) $ of
the integral Volterra equation (\ref{3.8}) of the second kind, where $%
v\left( t\right) \in H_{\alpha 0}^{1}\left[ 0,t_{1}\right] $ is a control
satisfying condition (\ref{3.0}) (see Definition \ref{D3.3}).
\end{remark}

\subsubsection{Singular case $\left( c,b_{2}\right) =0$}

If $\left( c,b_{2}\right) =0,$ then above consideration are not applicable,
because equation (\ref{3.8}) appears to be an Volterra equation

\begin{equation*}
w\left( t\right) =\int_{0}^{t}\left( S_{2}^{\ast }\left( t-\tau \right)
A_{2}^{\ast }c,b_{2}\right) U\left( \tau \right) d\tau .
\end{equation*}%
of the first kind with continuous kernel $K_{1}\left( t-\tau \right) =\left(
S_{2}^{\ast }\left( t-\tau \right) A_{2}^{\ast }c,b_{2}\right) $. However if
$c\in D\left( A_{2}^{\ast 2}\right) $ and $\left( A_{2}^{\ast
}c,b_{2}\right) \neq 0,$ then the proof of Theorem can be used, if
everywhere in the proof to replace the vector $c\in D\left( A_{2}^{\ast
}\right) $ by the vector $A_{2}^{\ast }c\in D\left( A_{2}^{\ast }\right) .$%
By this way the following results can be obtained.

\begin{definition}
\label{D3.2}Interconnected system (\ref{1.1})--(\ref{1.3}) is said to be
exact null-controllable on $[0,t_{1}]$ by distributions, if for each $%
x_{1}^{0}\in X_{1},~x_{20}\in X_{2}$ there exists a distribution
(generalized control) $u\left( \cdot \right) ,$ such that a mild solution\ $%
x_{1}\left( t,x_{1}^{0},v\left( \cdot \right) \right) $ of equation (\ref%
{1.1}) with a control $v\left( t\right) $ defined by (\ref{1.2}) satisfies
the condition
\begin{equation*}
x_{1}\left( t,x_{1}^{0},v\left( \cdot \right) \right) =0.
\end{equation*}
\end{definition}
\newpage
\begin{theorem}
\label{T3.3}If

\begin{itemize}
\item the family of exponents%
\begin{equation*}
\left\{ 1,\frac{b_{1j}}{\lambda _{j}}e^{-\lambda _{1j}t}\
,\,\,j=1,2,...\right\}
\end{equation*}%
is strongly minimal,

\item $c\in D\left( A_{2}^{\ast ^{2}}\right) ,~\left( c,b_{2}\right) =0~$and$%
~\left( A_{2}^{\ast }c,b_{2}\right) \neq 0,$
\end{itemize}

then interconnected equation (\ref{1.1})--(\ref{1.3}) is exact
null-controllable on $[0,t_{1}]$ by distributions.
\end{theorem}

\begin{proof}
Let $c\in D\left( A_{2}^{\ast ^{2}}\right) $~and $\left( A_{2}^{\ast
}c,b_{2}\right) \neq 0.$ Arguing as in the regular case, consider the
Volterra equation%
\begin{equation}
w\left( t\right) =v\left( t\right) -v\left( 0\right) =\left(
A_{2}c,b_{2}\right) U\left( t\right) +\int_{0}^{t}\left( S_{2}^{\ast }\left(
t-\tau \right) A_{2}^{\ast ^{2}}c,b_{2}\right) U\left( \tau \right) d\tau .
\label{3.10}
\end{equation}

%

Again because of $\left( A_{2}^{\ast }c,b_{2}\right) \neq 0$ and \ $c\in
D\left( A_{2}^{\ast ^{2}}\right) $ the kernel $\ \ K_{2}\left( t,\,\tau
\right) =\left( S_{2}^{\ast }\left( t-\tau \right) A_{2}^{\ast
^{2}}c,b_{2}\right) $ is continuous \cite{Hille&Philips, Krein}, hence
equation (\ref{3.10}) has a continuous solution $U_{1}\left( t\right) $ for
any continuous function \ $v\left( t\right) $ \cite{Tricomi}.

Since the function $w\left( t\right) $ is absolutely continuous with square
integrable derivative, from (\ref{3.10}) it follows that the function $%
U_{1}\left( t\right) $ is absolutely continuous, so as well as in the
regular case
\begin{eqnarray*}
w\left( t\right) &=&\left( A_{2}^{\ast }c,b_{2}\right) U_{1}\left( t\right)
\ +\int_{0}^{t}\left( S_{2}^{\ast }\left( t-\tau \right) A_{2}^{\ast
^{2}}c,b_{2}\right) U_{1}\left( \tau \right) d\tau = \\
&=&\int_{0}^{t}\left( S_{2}^{\ast }\left( t-\tau \right) A_{2}^{\ast
}c,b_{2}\right) U\left( \tau \right) d\tau ,
\end{eqnarray*}%
where $U_{1}\left( 0\right) =0,U\left( t\right) =\dot{U}_{1}\left( t\right) ~%
\mathrm{a.e.~for}~t\in \left[ 0,t_{1}\right] ,$and $U\left( \cdot \right)
\in L_{2}\left[ 0,t_{1}\right] \ $because of $\dot{v}\left( \cdot \right)
\in L_{2}\left[ 0,t_{1}\right] $. If to continue the integration by parts
for
\begin{equation*}
\int_{0}^{t}\left( S_{2}^{\ast }\left( t-\tau \right) A_{2}^{\ast
}c,b_{2}\right) U\left( \tau \right) d\tau ,
\end{equation*}%
we can only obtain (\ref{3.7}) for $u\left( t\right) $ which is understood
as a distribution (the first distributional(generalized) derivative of the
square integrable function $U\left( t\right) $ or the second distributional
derivative of the continuous function $U_{1}\left( t\right) ).$

This proves the theorem. 
\end{proof}

\textbf{\ }The same approach can be used, if $c\in D\left( A_{2}^{\ast
^{m}}\right) ,\left( A_{2}^{\ast ^{k}}c,b_{2}\right) =0,k=0,1,...,m-1,\left(
A_{2}^{\ast ^{m-1}}c,b_{2}\right) \neq 0$ for some $m\in \mathbb{N}.$

\textbf{Remark.} Using the Laplace Transform in (\ref{3.7}) we obtain for
sufficiently large $\func{Re}s$

\begin{equation}
W\left( s\right) =\left( c,\left( sI-A_{2}\right) ^{-1}b_{2}\right) U\left(
s\right)  \label{3.13}
\end{equation}

As much as $w\left( \cdot \right) \in H_{0t_{1}}^{1}\left[ 0,t_{1}\right] $
and \ $v\left( t\right) \equiv 0,t>t_{1},$ the Laplace Transform $W\left(
s\right) $ of the control $w\left( t\right) $ exists. Clearly, for each \ $%
W\left( s\right) $ equation (\ref{3.13}) is solvable with respect to $%
U\left( s\right) $ if and only if the scalar function $\left( c,\left(
sI-A_{2}\right) ^{-1}b_{2}\right) $ \ does not identically equal to zero for
sufficiently large $\func{Re}s\ .$ It surely holds true in regular case. $.$

\subsection{Dual controllability criterion}

Now consider the case $c\notin D(A_{2}^{\ast }).$

If $c\notin D(A_{2}^{\ast })$ then it is impossible to provide the existence
of continuous solution for equation (\ref{3.7}).

Obviously
\begin{equation}
\left( x_{2}^{\ast },S_{2}\left( t-\tau \right) x_{2}\right) =\left(
S_{2}^{\ast }\left( t-\tau \right) x_{2}^{\ast },x_{2}\right) ,~\forall
x_{2},x_{2}^{\ast }\in X_{2}.  \label{3.14}
\end{equation}%
Hence if $b_{2}\in D\left( A_{2}\right) ,$ then the function $\left(
c,S_{2}\left( t-\tau \right) b_{2}\right) $ is continuously differentiable
and
\begin{equation}
\frac{d}{dt}\left( c,S_{2}\left( t-\tau \right) b_{2}\right) =\left(
c,S_{2}\left( t-\tau \right) A_{2}b_{2}\right) ,~\tau \in \left[ 0,t\right] .
\label{3.15}
\end{equation}%
Using (\ref{3.14}) in (\ref{3.7}) we obtain that%
\begin{equation}
w\left( t\right) =\int_{0}^{t}\left( c,S_{2}\left( t-\tau \right)
b_{2}\right) u\left( \tau \right) d\tau  \label{3.16}
\end{equation}%
If $b\in D\left( A_{2}\right) ,$ then $\frac{\partial }{\partial \tau }%
K\left( t-\tau \right) =-\left( c,S_{2}\left( t-\tau \right)
A_{2}b_{2}\right) $ is continuous \cite{Hille&Philips, Krein}, hence if
absolutely continuous function $v\left( t\right) $ appears to be
continuously differentiable, there exists a continuous solution of equation (%
\ref{3.16})

If absolutely continuous function $v\left( t\right) $ is not continuously
differentiable, we consider the integral Volterra equation of the second kind 
\begin{equation}
\frac{w\left( t\right) }{\left( c,b_{2}\right) }=U\left( t\right)
+\int_{0}^{t}\frac{\left( c,S_{2}\left( t-\tau \right) A_{2}b_{2}\right) }{%
\left( c,b_{2}\right) }U\left( \tau \right) d\tau .  \label{3.17}
\end{equation}

Arguing as well as in the proof of the Theorem \ref{T3.2} we obtain the
validity of the following theorem:

\begin{theorem}
\label{T3.5}If

\begin{itemize}
\item the family of exponents%
\begin{equation}
\left\{ 1,\frac{b_{1j}}{\lambda _{j}}e^{-\lambda
_{1j}t},\,\,j=1,2,...\right\}  \label{3.18}
\end{equation}%
is strongly minimal,

\item $b_{2}\in D\left( A_{2}\right) \ $and$~\left( c,b_{2}\right) \neq 0,$
\end{itemize}

then interconnected equation (\ref{1.1})--(\ref{1.3}) is exact
null-controllable on $[0,t_{1}],$ where the control $u\left( t\right) $
satisfying (\ref{3.2}) is a square integrable first derivative of the
absolutely continuous solution $U\left( t\right) $ of the integral Volterra
equation of the second kind%
\begin{equation}
w\left( t\right) =\left( c,b_{2}\right) U\left( t\right) +\int_{0}^{t}\left(
c,S_{2}\left( t-\tau \right) A_{2}b_{2}\right) U\left( \tau \right) d\tau ,
\label{3.19}
\end{equation}
\end{theorem}

The generalization for the case $\left( c,b_{2}\right) =0$ is done as above
(see the previous subsection).

\subsection{Strong minimality of real exponential families.}

A direct proof of this fact for a given sequence of exponents can sometimes
be tough. Below we prove two lemmas which substantially facilitate the
establishment of the strong minimality for real exponential families.

\begin{lemma}
\label{L3.2} If $\mu _{n}>0$ , the series
\begin{equation}
\sum_{n=1}^{\infty }\frac{1}{\mu _{n}}\   \label{3.21}
\end{equation}%
converges and the Dirichle series
\begin{equation}
\sum_{n=1}^{\infty }\frac{e^{-\mu _{n}\alpha }}{\beta _{n}}\   \label{3.22}
\end{equation}%
converges for some $\alpha >0,$ then the sequence
\begin{equation}
\left\{ \beta _{n}e^{\mu _{n}t},n=1,2,...,t\in \left[ 0,t_{1}\right]
,\forall t_{1}>0\right\}  \label{3.23}
\end{equation}%
is strongly minimal.
\end{lemma}

\begin{proof}
Let $t_{1}=2t_{2}.$ Using results of \cite{Fattorini&Russel} one can show
that if the series $\sum_{n=1}^{\infty }\frac{1}{\mu _{n}}$ converges and $\
\lambda _{n+1}-\lambda _{n}\geq 1$, then the sequence
\begin{equation}
\left\{ e^{-\mu _{n}t},n=1,2,...,t\in \left[ 0,t_{2}\right] ,\forall
t_{2}>0\right\}  \label{3.24}
\end{equation}%
is minimal. Clearly the sequence
\begin{equation*}
\left\{ \beta _{n}e^{\mu _{n}t},n=1,2,...,t\in \left[ 0,t_{2}\right]
\right\} ,\forall t_{2}>0
\end{equation*}%
is also minimal. In virtue of Theorem 1.5 of \cite{Fattorini&Russel} for
each $\varepsilon >0$ there exists a positive constant $K_{\varepsilon }$
such that the sequence $\left\{ w_{n}\left( t\right) ,n=1,2,...,t\in \left[
0,t_{2}\right] \right\} $ biorthogonal to the sequence (\ref{3.23})
satisfies the condition%
\begin{equation*}
\left\Vert w_{n}\left( \cdot \right) \right\Vert <K_{\varepsilon
}e^{\varepsilon \mu _{n}},n=1,2,...,.
\end{equation*}%
Hence the sequence $\left\{ u_{n}\left( t\right) =\frac{1}{\beta _{n}}%
w_{n}\left( t_{2}-t\right) e^{-\mu _{n}t_{2}},n=1,2,...,t\in \left[ 0,t_{2}%
\right] \right\} $ is biorthogonal to the sequence (\ref{3.17}) and it
satisfies the condition%
\begin{equation}
\left\Vert u_{n}\left( \cdot \right) \right\Vert <\frac{1}{\beta _{n}}%
K_{\varepsilon }e^{\varepsilon \mu _{n}}e^{-\mu _{n}t_{2}}<\frac{1}{\beta
_{n}}K_{\varepsilon }e^{\varepsilon \mu _{n}},n=1,2,...,.  \label{3.25}
\end{equation}%
The positive constant $\varepsilon $ can be chosen such that $%
t_{2}-\varepsilon >\alpha .$

By the Minkowsky inequality and (\ref{3.25}) one can show that

$\sum\limits_{n=1}^{p}\sum\limits_{m=1}^{p}c_{n}e^{-\mu _{n}t_{2}}\left(
\int_{0}^{t_{2}}u_{n}\left( t\right) u_{m}\left( t\right) dt\right) e^{-\mu
_{m}t_{2}}c_{m}=$

$=\int_{0}^{t_{2}}\left( \sum\limits_{n=1}^{p}c_{n}e^{-\mu
_{n}t_{2}}u_{n}\left( t\right) dt\right) ^{2}dt\leq
\int_{0}^{t_{2}}\sum\limits_{n=1}^{p}\left\vert c_{n}\right\vert
^{2}\sum\limits_{n=1}^{p}\left\vert e^{-\mu _{n}t_{2}}u_{n}\left( t\right)
\right\vert ^{2}dt=$

$=\sum\limits_{n=1}^{p}\left\vert c_{n}\right\vert
^{2}\sum\limits_{n=1}^{p}\int_{0}^{t_{2}}\left\vert e^{-\mu
_{n}t_{2}}u_{n}\left( t\right) \right\vert ^{2}dt=$

$=\sum\limits_{n=1}^{p}\left\vert c_{n}\right\vert
^{2}\sum\limits_{n=1}^{p}e^{-2\mu _{n}t_{2}}\int_{0}^{t_{2}}\left\vert
u_{n}\left( t\right) \right\vert ^{2}dt\leq $

$\leq \sum\limits_{n=1}^{p}\left\vert c_{n}\right\vert
^{2}\sum\limits_{n=1}^{p}e^{-2\mu _{n}t_{2}}\left\Vert u_{n}\left( \cdot
\right) \right\Vert ^{2}\leq \sum\limits_{n=1}^{p}\left\vert
c_{n}\right\vert ^{2}K_{\varepsilon }^{2}\sum\limits_{n=1}^{p}\frac{1}{\beta
_{n}^{2}}e^{-2\mu _{n}\left( t_{2}-\varepsilon \right) }.$

It is well-known from the Dirichle series theory \cite{Leontiev}, that if
the Dirichle series (\ref{3.22}) converges, then the Dirichle series $%
\sum_{n=1}^{\infty }\frac{e^{-\mu _{n}t}}{\beta _{n}}$ converges for any $%
t\geq \alpha .$ Therefore according to theorem condition the series $%
\sum\limits_{n=1}^{\infty }\frac{1}{\beta _{n}}e^{-2\mu _{n}\left(
t_{2}-\varepsilon \right) }$ converges\footnote{%
The number $\alpha _{0}=\inf \left\{ \alpha \in \mathbb{R}:~\text{the series
(\ref{3.22}) converges}\right\} $ is said to be the abscissa of the
convergence of Dirichle series \cite{Leontiev}.} for any $t_{2},\varepsilon
,t_{2}>\varepsilon +\alpha ,$ so the same holds true for the series $%
\sum\limits_{n=1}^{\infty }\frac{1}{\beta _{n}^{2}}e^{-2\mu _{n}\left(
t_{2}-\varepsilon \right) }$ .~Hence $\sum\limits_{n=1}^{p}\frac{1}{\beta
_{n}^{2}}e^{-2\mu _{n}\left( t_{2}-\varepsilon \right) }\leq M$, \noindent
where $M$ is a positive constant, so
\begin{equation}
\sum\limits_{n=1}^{p}\sum\limits_{m=1}^{p}c_{n}\ e^{-\mu _{n}t_{2}}\left(
\int_{0}^{t_{2}}u_{n}\left( t\right) u_{m}\left( t\right) dt\right) e^{-\mu
_{m}t_{2}}c_{m}\leq K_{\varepsilon }^{2}M\sum\limits_{n=1}^{p}\left\vert
c_{n}\right\vert ^{2}  \label{3.26}
\end{equation}%
for every finite sequence $\left\{ c_{1},c_{2},...,c_{p}\right\} .$

The sequence
\begin{equation*}
\left\{ h_{n}\left( t\right) ,n=1,2,...,\right\} ,t\in \left[ 0,t_{1}\right]
,
\end{equation*}%
where\footnote{%
Recall, that $t_{1}=2t_{2}$.}
\begin{equation}
h_{n}\left( t\right) =\left\{
\begin{array}{cc}
e^{-\mu _{n}t_{2}}u_{n}\left( t-t_{2}\right) , & t\in \left[ t_{2},2t_{2}%
\right] , \\
0, & t\in \left[ 0,t_{2}\right) ,%
\end{array}%
\right. n=1,2,...,  \label{3.27}
\end{equation}%
is the biorthogonal to the sequence
\begin{equation*}
\left\{ \beta _{n}e^{\mu _{n}t},n=1,2,...,t\in \left[ 0,t_{1}\right]
\right\} .
\end{equation*}%
Indeed,

$\int_{0}^{t_{1}}\beta _{n}e^{\mu _{n}t}h_{m}\left( t\right)
dt=\int_{t_{2}}^{2t_{2}}\beta _{n}e^{\mu _{n}t}e^{-\mu _{m}t_{2}}u_{m}\left(
t-t_{2}\right) dt=$

$=e^{\left( \mu _{n}-\mu _{m}\right) t_{2}}\int_{0}^{t_{2}}\beta _{n}e^{\mu
_{n}t}u_{m}\left( \tau \right) d\tau =\delta _{nm},n,m=1,2,...,$ where $%
\delta _{nm},n,m=1,2,...,$ is the Kroneker Delta.

Further we have

\begin{eqnarray}
\int_{0}^{t_{1}}h_{n}\left( t\right) h_{m}\left( t\right) dt &=&e^{-\mu
_{n}t_{2}}\left( \int_{t_{2}}^{2t_{2}}u_{n}\left( t-t_{2}\right) u_{m}\left(
t-t_{2}\right) dt\right) e^{-\mu _{m}t_{2}}=  \notag \\
&=&e^{-\mu _{n}t_{2}}\left( \int_{0}^{t_{2}}u_{n}\left( t\right) u_{m}\left(
t\right) dt\right) e^{-\mu _{m}t_{2}},  \label{3.28}
\end{eqnarray}

\noindent and so it follows from (\ref{3.26})--(\ref{3.28}), that
\begin{equation*}
\sum_{n=1}^{p}\sum_{m=1}^{p}c_{n}\left( \int_{0}^{t_{1}}h_{n}\left( t\right)
h_{m}\left( t\right) dt\right) c_{m}\leq K_{\varepsilon
}^{2}M\sum_{n=1}^{p}\left\vert c_{n}\right\vert ^{2}.
\end{equation*}%
Hence \cite{Kaczmarz&Steinhaus}
\begin{equation*}
\sum_{n=1}^{p}\sum_{m=1}^{p}c_{n}\left( \int_{0}^{t_{1}}\beta _{n}e^{\mu
_{n}t}\beta _{m}e^{\mu _{m}t}\right) c_{m}d\tau \geq \gamma
\sum_{n=1}^{p}\left\vert c_{n}\right\vert ^{2},p=1,2,...,
\end{equation*}%
for every finite sequence $\left\{ c_{1},c_{2},...,c_{p}\right\} ,$ where $%
\gamma =\frac{1}{K_{\varepsilon }^{2}M}>0.$ It proves that the sequence
\begin{equation*}
\left\{ \beta _{n}e^{\mu _{n}t},t\in \left[ 0,t_{1}\right]
,~n=1,2,...\right\}
\end{equation*}%
is strongly minimal for any $t_{1}>0$.~QED
\end{proof}

\begin{lemma}
\label{L3.3} If conditions of Lemma \ref{L3.1} hold, then the sequence
\begin{equation}
\left\{ 1,\beta _{n}e^{\mu _{n}t},n=1,2,...,t\in \left[ 0,t_{1}\right]
,\forall t_{1}>0\right\}  \label{3.29}
\end{equation}%
is also strongly minimal.
\end{lemma}

\begin{proof}
One can write family (\ref{3.29}) by%
\begin{equation}
g_{n}e^{v_{n}t},n=0,1,2,...,  \label{3.30}
\end{equation}%
where
\begin{eqnarray*}
v_{n} &=&\left\{
\begin{array}{cc}
\alpha & n=0 \\
\mu _{n}+\alpha & n=1,2,...%
\end{array}%
\right. , \\
g_{n} &=&\left\{
\begin{array}{cc}
1, & n=0 \\
\beta _{n} & n=1,2,...%
\end{array}%
\right. .
\end{eqnarray*}%
\newline
One can see that the family (\ref{3.30}) is the family of the form (\ref%
{3.23}), and from the convergence of the series $\sum\limits_{n=1}^{\infty }%
\frac{1}{\beta _{n}}e^{-\mu _{n}\left( t_{2}-\varepsilon \right) }$ it
follows that the series $\sum\limits_{n=1}^{p}\frac{1}{g_{n}^{2}}e^{-2\left(
\mu _{n}+\alpha \right) \left( t_{1}-\varepsilon \right) }$ also converges
for any $t_{2},\varepsilon ,t_{1}>\varepsilon .$ Hence in accordance with
Lemma \ref{L3.2} the sequence
\begin{equation}
\left\{ g_{n}e^{v_{n}t}=\left\{
\begin{array}{cc}
e^{\alpha t}, & n=0, \\
g_{n}e^{\left( \mu _{n}+\alpha \right) t}, & n=1,2,...,%
\end{array}%
\right. ,t\in \left[ 0,t_{2}\right] \right\}  \label{3.31}
\end{equation}%
is strongly minimal for any $t_{1}>0$, i.e. there exists a constant ~$\gamma
>0$ such that

\begin{equation}
\left\Vert \sum_{n=0}^{p}c_{n}g_{n}e^{v_{n}t}\right\Vert ^{2}\geq \gamma
\left\Vert \sum_{n=0}^{p}c_{n}\right\Vert ^{2}.  \label{3.32}
\end{equation}

As much as $\alpha >0,$ we obtain by (\ref{3.31})--(\ref{3.32}), that

$\left\Vert \left( c_{0}\ +\sum_{n=1}^{p}c_{n}g_{n}e^{\mu _{n}t}\right)
\right\Vert ^{2}=\left\Vert e^{-\alpha
t}\sum_{n=0}^{p}c_{n}g_{n}e^{v_{n}t}\right\Vert ^{2}\geq $

$\geq e^{-2\alpha t_{1}}\left\Vert \left( c_{0}\
+\sum_{n=1}^{p}c_{n}g_{n}e^{\mu _{n}t}\right) \right\Vert ^{2}\geq \gamma
_{\alpha }\left\Vert \sum_{n=0}^{p}c_{n}\right\Vert ^{2}, \forall p=1,2,...,$

\noindent where $\gamma _{\alpha }=\gamma e^{-2\alpha t_{1}}>0.$

It proves the Lemma.
\end{proof}

\section{Examples. Exact null controllability of interconnected Heat
Equation and Wave Equation by force motion}

We consider the heat equation with force motion
\begin{align}
y_{t}^{\prime }& =y_{xx}^{\prime \prime }+b_{1}\left( x\right) v\left(
t\right) ,0\leq t\leq t_{1},~0\leq x\leq \pi ,  \label{4.1} \\
y\left( 0,t\right) & =y\left( \pi ,t\right) =0,0\leq t\leq t_{1},
\label{4.2} \\
\ y\left( x,0\right) & =\varphi _{0}\left( x\right) ,~0\leq x\leq \pi ,
\label{4.3}
\end{align}%
governed by a control $u\left( \cdot \right) $ of the wave equation%
\begin{align}
z_{tt}^{\prime \prime }-z_{xx}^{\prime \prime }& =b_{2}\left( x\right)
u\left( t\right) ,~0\leq t\leq t_{1},0\leq x\leq \pi ,  \label{4.5} \\
z\left( 0,t\right) & =z\left( \pi ,t\right) =0,0\leq t\leq t_{1},
\label{4.6} \\
z\left( x,0\right) & =\psi _{0}\left( x\right) ,~z_{t}^{\prime }\left(
x,0\right) =\psi _{1}\left( x\right) ,~0\leq x\leq \pi ,  \label{4.7}
\end{align}

Here $\varphi _{0},~\psi ^{j},j=1,2,~$and$~~b_{1}\left( x\right)
,b_{2}\left( x\right) \ $belong to $L_{2}\left[ 0,\pi \right] .$

Let $H^{2}\left[ 0,\pi \right] ,~H_{0}^{1}\left[ 0,\pi \right] $ be Sobolev
spaces (see \cite{Hutson&Pum} for definitions of the spaces $H^{m}\left[ a,b%
\right] ,~H_{0}^{m}\left[ a,b\right] ,a,b\in \mathbb{R}.$)

Heat equation (\ref{4.1})--(\ref{4.3}) can be written in the semigroup
framework (\ref{1.3}), where $X_{1}=L_{2}\left[ 0,\pi \right] ,$ the
operator $A_{1}$ is defined by the differential operator $A_{1}=\frac{d^{2}}{%
dx^{2}}\ $ with the domain \footnote{$D\left( A_{1}\right) $ can also be
defined by \cite{Balakrishnan}
\begin{equation*}
D\left( A_{1}\right) =\left\{
\begin{array}{c}
y\left( \cdot \right) ,y^{\prime }\left( \cdot \right) \in AC\left[ 0,\pi %
\right] , \\
y^{\prime \prime }\left( \cdot \right) \in L_{2}\left[ 0,\pi \right]
,y\left( 0\right) =y\left( \pi \right) =0,%
\end{array}%
\right\}
\end{equation*}%
where $AC\left[ 0,\pi \right] $ is the set of absolutely continuous on $%
\left[ 0,\pi \right] $ functions.
\par
According to \cite{Hutson&Pum}, $H_{0}^{1}\left[ 0,\pi \right] =\left\{
y\left( \cdot \right) \in AC\left[ 0,\pi \right] ,y\left( 0\right) =y\left(
\pi \right) =0\right\} .$}
\begin{equation}
D\left( A_{1}\right) =H^{2}\left[ 0,\pi \right] \cap H_{0}^{1}\left[ 0,\pi %
\right] .  \label{4.8}
\end{equation}

Denote: $z_{1}\left( x,t\right) =z\left( x,t\right) ,z_{2}\left( x,t\right)
=z_{t}^{\prime }\left( x,t\right) ,$%
\begin{equation*}
X_{2}=H_{0}^{1}\left[ 0,\pi \right] \times L_{2}\left[ 0,\pi \right]
=\left\{ \left( z_{1},z_{2}\right) :z_{1}\in H_{0}^{1}\left[ 0,\pi \right]
,z_{2}\in L_{2}\left[ 0,\pi \right] \right\}
\end{equation*}%
with the scalar product of $\left(
\begin{array}{c}
z_{1} \\
z_{2}%
\end{array}%
\right) ,\left(
\begin{array}{c}
y_{1} \\
y_{2}%
\end{array}%
\right) \in X_{2}$ defined by

$\left( \left(
\begin{array}{c}
z_{1} \\
z_{2}%
\end{array}%
\right) ,\left(
\begin{array}{c}
y_{1} \\
y_{2}%
\end{array}%
\right) \right) =\int\limits_{0}^{\pi }\left( z_{1}^{\prime }\left( x\right)
y_{1}^{\prime }\left( x\right) +z_{2}\left( x\right) y_{2}\left( x\right)
\right) dx.$

Wave equation (\ref{4.5})--(\ref{4.7}) can be written in the semigroup
framework (\ref{1.3}), where $x_{2}\left( t\right) =\left(
\begin{array}{c}
z_{1}\left( x,t\right) \\
z_{2}\left( x,t\right)%
\end{array}%
\right) =\left(
\begin{array}{c}
z\left( x,t\right) \\
z_{t}^{\prime }\left( x,t\right)%
\end{array}%
\right) ,b_{2}=\left(
\begin{array}{c}
0 \\
b_{2}\left( x\right)%
\end{array}%
\right) .$

The operator $A_{2}$ is defined by the matrix differential operator $%
A_{2}=\left(
\begin{array}{cc}
0 & 1 \\
\frac{d^{2}}{dx^{2}} & 0%
\end{array}%
\right) $ with the domain

\begin{equation}
D\left( A_{2}\right) =\left( z_{1}\left( x\right) ,z_{2}\left( x\right)
\right) \in \left( H^{2}\left[ 0,\pi \right] \cap H_{0}^{1}\left[ 0,\pi %
\right] \right) \times H_{0}^{1}\left[ 0,\pi \right] ,  \label{4.10}
\end{equation}%
\begin{equation}
A_{2}z=A_{2}\left(
\begin{array}{c}
z_{1} \\
z_{2}%
\end{array}%
\right) =\left(
\begin{array}{c}
z_{2} \\
z_{1}^{\prime \prime }%
\end{array}%
\right) ,  \label{4.11}
\end{equation}

A smooth distributed control $v\left( t\right) ~$in the force motion term of
heat equation (\ref{4.1})--(\ref{4.3})can be considered as an observation

\begin{equation}
v\left( t\right) =\int_{0}^{\pi }\left( c_{1}^{\prime }\left( x\right)
z_{x}^{\prime }\left( x,t\right) +c_{2}\left( x\right) z_{t}^{\prime }\left(
x,t\right) \right) dx,~0\leq t\leq t_{1}\pi ,c_{2}\left( \cdot \right) \in
L_{2}\left[ 0,\pi \right]  \label{4.4}
\end{equation}%
of wave equation (\ref{4.5})--(\ref{4.7}), where $c\left( \cdot \right)
=\left( c_{1}\left( \cdot \right) ,c_{2}\left( \cdot \right) \right) \in
H_{0}^{1}\left[ 0,\pi \right] \times L_{2}\left[ 0,\pi \right] .$

Therefore the results obtained in the previous section can be applied.

\subsubsection{Controllability conditions}

The eigenvalues $\lambda _{n}\ $and corresponding eigenvectors $\varphi
_{n},n=1,2,...$of the operator $A_{1.}$ are obtained by $\lambda
_{n}=-n^{2},\varphi _{n}\ \sin nx,n=1,2,...,$ and as much as the operator $%
A_{1}$ is selfadjoint, the eigenvectors $\varphi _{n}$~and~the eigenvectors $%
\psi _{n},n=1,2,...,$ of the adjoint operator $A_{1}^{\ast }$ are the same.

Obviously, the sequence $\sin nx,n=1,2,...\ $of eigenvectors of the operator
$A_{1}~$(or $A_{1}^{\ast }$) \ forms a Riesz basic in $X_{1}.$

Denote $b_{1n}=\int_{0}^{\pi }b_{1}\left( x\right) \sin nx,n=1,2,...,.\ $

In accordance with Theorem \ref{T3.2} to establish the conditions of the
exact null controllability of interconnected system under consideration we
should prove that the family
\begin{equation}
\left\{ 1,\frac{b_{1n}}{n^{2}}e^{n^{2}t},n=1,2,\dots ,...,t\in \lbrack
0,t_{1}]\right\}  \label{4.12}
\end{equation}%
is strongly minimal.

In our case $\mu _{n}=\ \lambda _{n}=\ n^{2},n=1,2,...,$ the series $%
\sum_{n=1}^{\infty }\frac{1}{\mu _{n}}=\sum_{n=1}^{\infty }\frac{1}{n^{2}}$
converges.

Hence in accordance with Theorems \ref{T3.1}--\ref{T3.2} and Lemma \ref{L3.1}
we obtain the validity of the following theorem:

\begin{theorem}
\label{T4.1}If the series%
\begin{equation}
\sum_{n=1}^{\infty }\frac{n^{2}}{b_{1n}}e^{-n^{2}\alpha }  \label{4.13}
\end{equation}
converges for some $\alpha >0,$ then equation (\ref{4.1})--(\ref{4.3}) is
exact null-controllable on $\left[ 0,t_{1}\right] ,\forall t_{1}>0,$ by
smooth controls $v\left( \cdot \right) \in H_{\alpha 0}^{1}\left[ 0,t_{1}%
\right] $.
\end{theorem}

\subsection{Regular case}

%

\begin{theorem}
\label{T4.2}If

\begin{enumerate}
\item series (\ref{4.13}) converges for some $\alpha >0,$

\item $b_{2}\left( \cdot \right) ,c_{2}\left( \cdot \right) \in L_{2}\left[
0,\pi \right] ,$

\item $\int_{0}^{\pi }c_{2}\left( x\right) b_{2}\left( x\right) dx\neq 0,$
\end{enumerate}

then interconnected system equation (\ref{4.1})--(\ref{4.3}),(\ref{4.5})--(%
\ref{4.7}),interconnected by (\ref{4.4}) is exact null-controllable on $%
\left[ 0,t_{1}\right] ,\forall t_{1}>0.$
\end{theorem}

\begin{proof}
We have here $c\left( \cdot \right) =\left(
\begin{array}{c}
c_{1}\left( \cdot \right) \\
c_{2}\left( \cdot \right)%
\end{array}%
\right) ,$ $b_{2}=\left(
\begin{array}{c}
0 \\
b_{2}\left( \cdot \right)%
\end{array}%
\right) .$ Therefore $b_{2}\in X_{2}$ is equivalent to $b_{2}\left( \cdot
\right) \in L_{2}\left[ 0,\pi \right] ,~$and$~\left( c,b_{2}\right)
=\int_{0}^{\pi }c_{2}\left( x\right) b_{2}\left( x\right)
dx,~A_{2}b_{2}=A_{2}\left(
\begin{array}{c}
0 \\
b_{2}\left( \cdot \right)%
\end{array}%
\right) =\left(
\begin{array}{c}
b_{2}\left( \cdot \right) \\
0%
\end{array}%
\right) ,$ so $b_{2}=\left(
\begin{array}{c}
0 \\
b_{2}\left( x\right)%
\end{array}%
\right) \in D\left( A_{2}\right) ~\ \ $for \ any $b_{2}\left( \cdot \right)
\in L_{2}\left[ 0,\pi \right] ~$and $\left( c,b_{2}\right) \neq 0.$ The
theorem follows from Theorem \ref{T3.5}. We have here the regular case.

\textbf{Remark. }The theorem assertion is valid for any function $%
c_{1}\left( \cdot \right) \in H_{0}^{1}\left[ 0,\pi \right] .$
\end{proof}

For example, conditions of Theorem \ref{T4.2} hold true for $b_{1}\left(
x\right) =b_{2}\left( x\right) =x,c_{2}\left( x\right) =1,x\in \left[ 0,\pi %
\right] .$ In this case $b_{1n}=\int_{0}^{\pi }x\sin nxdx=\frac{\left(
-1\right) ^{n+1}\pi }{n},$ the series $\sum\limits_{n=1}^{\infty }\frac{n^{2}%
}{b_{1n}}e^{-n^{2}\alpha }=\sum\limits_{n=1}^{\infty }\frac{\left( -1\right)
^{n+1}n^{3}}{\pi }e^{-n^{2}\alpha }$ converges for any $\alpha
>0;~b_{2}\left( \cdot \right) \in L_{2}\left[ 0,\pi \right] ,~$and $%
\int_{0}^{\pi }c_{2}\left( x\right) b_{2}\left( x\right) dx=\int_{0}^{\pi
}xdx=\allowbreak \frac{\pi ^{2}}{2}.$

\subsection{Singular case}

Let $c_{2}\left( \cdot \right) ,b_{2}\left( \cdot \right) \in L_{2}\left[
0,\pi \right] ,~$but $\int_{0}^{\pi }c_{2}\left( x\right) b_{2}\left(
x\right) dx=0.$ In this case $\left( c,b_{2}\right) =0,$ so Theorems \ref%
{T3.2} or \ref{T3.5} are not applicable, ant it is impossible to provide to
exact null controllability of interconnected system being considered.

\begin{theorem}
\label{T4.3}If

\begin{enumerate}
\item series (\ref{4.13}) converges for some $\alpha >0,$

\item Either $b_{2}\left( \cdot \right) \in H^{2}\left[ 0,\pi \right] \cap
H_{0}^{1}\left[ 0,\pi \right] ~$or $c_{1}\left( \cdot \right) ,c_{2}\left(
\cdot \right) \in H^{2}\left[ 0,\pi \right] \cap H_{0}^{1}\left[ 0,\pi %
\right] ,$

\item $\int_{0}^{\pi }c_{1}\left( x\right) b_{2}\left( x\right) dx\neq 0,$
\end{enumerate}

then system (\ref{4.1})--(\ref{4.3}),(\ref{4.5})-(\ref{4.7}), interconnected
by (\ref{4.4}) is exact null-controllable on $\left[ 0,t_{1}\right] $,$%
\forall t_{1}>0$ by distributions.
\end{theorem}

\begin{proof}
We have here $A_{2}^{\ast }=\left(
\begin{array}{cc}
0 & \frac{d^{2}}{dx^{2}} \\
1 & 0%
\end{array}%
\right) ,~A_{2}^{\ast }\left(
\begin{array}{c}
z_{1} \\
z_{2}%
\end{array}%
\right) =\left(
\begin{array}{c}
z_{2}^{\prime \prime } \\
z_{1}%
\end{array}%
\right) ,$ $c\left( \cdot \right) =\left(
\begin{array}{c}
c_{1}\left( \cdot \right) \\
c_{2}\left( \cdot \right)%
\end{array}%
\right) ,$ $b_{2}=\left(
\begin{array}{c}
0 \\
b_{2}\left( \cdot \right)%
\end{array}%
\right).$ Therefore

if $b_{2}\left( \cdot \right) \in H^{2}\left[ 0,\pi \right] \cap H_{0}^{1}%
\left[ 0,\pi \right] ,$ then

$A_{2}b_{2}$ $=A_{2}\left(
\begin{array}{c}
0 \\
b_{2}\left( \cdot \right)%
\end{array}%
\right) =\left(
\begin{array}{c}
b_{2}\left( \cdot \right) \\
0%
\end{array}%
\right) ,~A_{2}^{2}b_{2}=\left(
\begin{array}{c}
0 \\
b_{2}^{\prime \prime }\left( \cdot \right)%
\end{array}%
\right),$ so

$b_{2}=\left(
\begin{array}{c}
0 \\
b_{2}\left( \cdot \right)%
\end{array}%
\right) \in D^{2}\left( A_{2}\right) $ ;

if $c_{1}\left( \cdot \right) ,c_{2}\left( \cdot \right) \in H^{2}\left[
0,\pi \right] \cap H_{0}^{1}\left[ 0,\pi \right] ,$ then

$A_{2}^{\ast }c_{2}=A_{2}^{\ast }\left(
\begin{array}{c}
c_{1}\left( \cdot \right) \\
c_{2}\left( \cdot \right)%
\end{array}%
\right) =\left(
\begin{array}{c}
c_{2}^{\prime \prime }\left( \cdot \right) \\
c_{1}\left( \cdot \right)%
\end{array}%
\right), A_{2}^{\ast 2}c=A_{2}^{\ast }\left(
\begin{array}{c}
c_{2}^{\prime \prime }\left( \cdot \right) \\
c_{1}\left( \cdot \right)%
\end{array}%
\right) =\left(
\begin{array}{c}
c_{1}^{\prime \prime }\left( \cdot \right) \\
c_{2}^{\prime \prime }\left( \cdot \right)%
\end{array}%
\right), $ \noindent so $c\in D^{2}\left( A_{2}^{\ast }\right) $.

In both cases we have $\left( c,b_{2}\right) =0,\left( c,A_{2}b_{2}\right)
=\int_{0}^{\pi }c_{1}\left( x\right) b_{2}\left( x\right) dx,$ so according
to the third condition of the theorem $\left( c,A_{2}b_{2}\right) \neq 0$.
Hence the theorem follows from Theorems \ref{T3.3}. We have here the
singular case.
\end{proof}

For example, conditions of Theorem \ref{T4.3} hold true for $b_{1}\left(
x\right) =b_{2}\left( x\right) =x,c_{1}\left( x\right) =x\left( \pi
-x\right) ,c_{2}\left( x\right) =0,x\in \left[ 0,\pi \right] .$

\section{Conclusion}

Exact null-controllability conditions for two interconnected abstract
control equations (\ref{1.1})--(\ref{1.3}) governed by a control $u\left(
t\right) $ of equation (\ref{1.3}) are obtained.

Of course, these results can be extended for series of a number
interconnected equations, governed by a control of the last one.

The case of simply eigenvalues has been considered for the sake of
simplicity only.

The main problems allowing to obtain controllability results of the paper
are :

\begin{enumerate}
\item Establishing of the exact null-controllability of equation (\ref{1.1})
by smooth control.

\item Solvability conditions of Volterra integral equation (\ref{3.16}) of
the first kind and of the convolution type.
\end{enumerate}

%

Both these problems are independent on each other. The mutual independence
of these problems allow us to use the abstract approach developed in the
given paper for investigation of various control problems for interconnected
systems contained equations of a different structure. For example, equation (%
\ref{1.1}) may be a parabolic control equation, governed by force motion
control, and equation (\ref{1.3}) may be a linear differential control
system with delays \cite{Bellman&Cooke,Hale}, governed by force motion
control, and so on\footnote{%
As is known the author the existing results in this field are devoted to
couple PDE's.}.

The singular case does not seem to be essential (in our opinion), because
there are a lot of practical situations, for which equation (\ref{1.1})--(%
\ref{1.3}) and (\ref{4.5})--(\ref{4.7}) are given, and need to decide, how
to connect them. It means, that if the case $\left(c,b_{2}\right) =0$
occurs, one can always choose other vector $c,$ slightly different from the
first one, such that $\left( c,b_{2}\right) \neq 0.$

The exact null-controllability for interconnected heat--wave equations is
considered as illustrative example only. Surely the controllability problems
for many kinds of control PDE's have been extensively investigated in last
years.

In our private opinion, a great majority of them can be investigated by the
abstract approach presented in the given paper.

\label{References}

\end{document}